\documentclass{amsart}

\usepackage{amsmath}
\usepackage{amssymb}
\usepackage{mathrsfs}
\usepackage{amsthm}
\usepackage{enumerate}
\usepackage{color}
\usepackage{verbatim}

\newtheorem{thm}{Theorem}[section]

\newtheorem{cor}[thm]{Corollary}
\newtheorem{lem}[thm]{Lemma}

\theoremstyle{definition}

\newtheorem{rmk}[thm]{Remark}

\numberwithin{equation}{section}

\title[Lipschitz and commutator estimates for commuting tuples]{Weak type operator Lipschitz and commutator estimates for commuting tuples}

\author{M. Caspers, F. Sukochev, D. Zanin}
\date{\today, {\it MSC2000}: 47B10, 47L20, 47A30.\\ {\it Support acknowledgements.} First author: This project has received funding from the European Union's Horizon 2020 research and innovation programme under grant agreement No. 702139. The work of the second and third named authors is supported by the ARC}

\address{M. Caspers, Mathematisch Instituut,
Budapestlaan 6, 3584 CD, Utrecht, The Netherlands}
\email{m.p.t.caspers@uu.nl}
\address{F. Sukochev, D. Zanin, School of Mathematics and Statistics, UNSW, Kensington 2052, NSW, Australia}
 \email{f.sukochev@unsw.edu.au}
\email{d.zanin@unsw.edu.au}

\begin{document}

\begin{abstract}
Let $f: \mathbb{R}^d \to\mathbb{R}$ be a Lipschitz function. If $B$ is a bounded self-adjoint operator and if $\{A_k\}_{k=1}^d$ are commuting bounded self-adjoint operators such that $[A_k,B]\in L_1(H),$ then
$$\|[f(A_1,\cdots,A_d),B]\|_{1,\infty}\leq c(d)\|\nabla(f)\|_{\infty}\max_{1\leq k\leq d}\|[A_k,B]\|_1,$$
where $c(d)$ is a constant independent of $f$, $\mathcal{M}$ and $A,B$ and $\|\cdot\|_{1,\infty}$ denotes the weak $L_1$-norm.

If $\{X_k\}_{k=1}^d$ (respectively, $\{Y_k\}_{k=1}^d$) are commuting bounded self-adjoint operators such that $X_k-Y_k\in L_1(H),$ then
$$\|f(X_1,\cdots,X_d)-f(Y_1,\cdots,Y_d)\|_{1,\infty}\leq c(d)\|\nabla(f)\|_{\infty}\max_{1\leq k\leq d}\|X_k-Y_k\|_1.$$
\end{abstract}

\maketitle

\section{Introduction}

Let $f: \mathbb{R} \rightarrow \mathbb{R}$ be a Lipschitz function. Let $\mathcal{M}$ be a semi-finite von Neumann algebra and let $\mathcal{M}_{sa}$ be its self-adjoint part. This paper deals with differentiability properties of (multi-dimensional versions of) the mapping
\begin{equation}\label{Eqn=SmoothnessMap}
\mathcal{M}_{sa} \ni A \mapsto f(A).
\end{equation}
The interest in such differentiability problems comes from very diverse directions: (i) the mapping  \eqref{Eqn=SmoothnessMap} relates strongly to perturbations of commutators, (ii) there is a prolific series of papers devoted to differentiability and   Lipschitz properties of \eqref{Eqn=SmoothnessMap}, (iii)  the map \eqref{Eqn=SmoothnessMap} relates to  Connes' non-commutative geometry and in particular the spectral action, see \cite{CM}, \cite{Skripka}, \cite{Suij}.

The roots of the results of this paper can be traced back to a problem of Krein \cite{Krein}   which led to a remarkable diversity of papers concerning double operator integrals and Schur multipliers.  The original Krein problem   asks if for a function $f$  being Lipschitz implies that it is operator Lipschitz, meaning that \eqref{Eqn=SmoothnessMap} is Lipschitz for the uniform norm on $\mathcal{M}_{sa}$. Krein's question is very natural but it was shown  that it has a negative answer \cite{Far8}, unless one imposes stricter differentiability assumptions on $f$ (like belonging to certain Besov or Sobolev spaces), see \cite{APPS}, \cite{BiSo1}, \cite{PellerHankel} to name just a few. Contributions to the problem were made by various people including Davies \cite{Davies}, Kato \cite{Kato} and Kosaki \cite{Kosaki} who found positive and negative results (under suitable conditions) for the analogue of Krein's problem for $L_p$-norms.

With the development of double operator integrals (see e.g.  \cite{BirSolDOI}, \cite{PSJFA04},  \cite{PSW}) significant steps forward were made on Lipschitz and differentiability properties of the mapping \eqref{Eqn=SmoothnessMap}, which were shown to be equivalent to various commutator estimates (see \cite{BirSol89}, \cite[Theorem 2.2]{DDPS1}). In turn this led to questions on the behavior of certain Schur multipliers and related double operator integrals.

Finding estimates -- even if they are non-optimal -- for norms of Schur multipliers is a highly non-trivial task. The hard part is that Schur multipliers acting on  $L_\infty$-spaces (or just matrix algebras) can often be estimated using Stinespring dilations, see e.g. \cite{PisierBook}. However, if one considers Schur multipliers on $L_p$-spaces  this tool is inapplicable. Therefore, in order to attack Krein's problem for $L_p$-spaces, $p\not = 1, \infty$ we are forced to introduce new techniques.

 A corner stone result was obtained in  \cite{PotapovSukochev2} (see also \cite{HNVW}): it was shown by D. Potapov and the second named author that the mapping \eqref{Eqn=SmoothnessMap} is Lipschitz continuous with respect to the $L_p$-norm, $1 < p < \infty$. As \cite{PotapovSukochev2} involves an application of the vector valued Marcinkiewicz multiplier theorem (due to Bourgain) it was not clear what the optimal non-commutative Lipschitz constants are.  A sharp  estimate for $L_p$-spaces was found in \cite{CMPS}. However in the category of symmetric spaces the question whether the so-called weak-$(1,1)$ estimate holds remained open.

 A first result in this weak-$(1,1)$ direction was obtained by Nazarov and Peller \cite{NazarovPeller} who  proved it in the special case that $A - B$ has rank 1. In the same paper a question concerning validity of this result for an arbitrary trace class perturbation  $A-B$ was posed.  A full answer for $f$ being the absolute value map was obtained in \cite{CPSZ} using positive definite Schur multipliers and triangular truncations. In \cite{CPSZ2} this result was extended to all Lipschitz functions. The result is ultimate for the functions of 1 variable: it is optimal within the category of symmetric spaces and it implies all other known estimates on perturbations of commutators and  Lipschitz functions obtained before \cite{CMPS},   \cite{CPSZ}, \cite{Davies},  \cite{DDPS1}, \cite{DDPS2}, \cite{Kato}, \cite{Kosaki}, \cite{NazarovPeller},  \cite{PotapovSukochev2}. The key ingredient of the proof in \cite{CPSZ2} is a new connection with non-commutative Calder\'on-Zygmund theory and in particular with the main result from Parcet's fundamental paper \cite{Parcet} (see also the recent paper by Cadilhac \cite{Cad} for a substantially shorter proof). \hyphenation{multi-pliers}\hyphenation{ana-ly-sis}

In this paper we focus on multi-dimensional (or multi-variable) Lipschitz estimates for the mapping \eqref{Eqn=SmoothnessMap} which naturally includes a version of the Nazarov-Peller problem for normal operators.
This study is deeply connected with that of classical Fourier multipliers. In particular, the dimension dependence of classes of multipliers as Bochner-Riesz multipliers, Riesz multipliers, (directional) Hilbert transforms et cetera, has been an important theme of research (we refer to Grafakos's book \cite{Grafakos} with ample such results).   Therefore, it is natural to look at the higher dimensional behavior of \eqref{Eqn=SmoothnessMap}.  Some results were obtained in \cite{KPSS} and \cite{CMPS}. However, the results in these papers are not optimal.   In this paper we obtain the following.

\begin{thm}\label{final theorem} For every Lipschitz function $f:\mathbb{R}^d\to\mathbb{R}$ and for every collection $\mathbf{A}=\{A_k\}_{k=1}^d\subset B(H)$ of commuting self-adjoint operators such that $[A_k,B]\in L_1(H),$ we have
$$\|[f(\mathbf{A}),B]\|_{1,\infty}\leq c(d)\|\nabla(f)\|_{\infty}\cdot\max_{1\leq k\leq d}\|[A_k,B]\|_1.$$
For every Lipschitz function $f:\mathbb{R}^d\to\mathbb{R}$ and for every collections $\mathbf{X}=\{X_k\}_{k=1}^d\subset B(H),$ $\mathbf{Y}=\{Y_k\}_{k=1}^d\subset B(H)$ of commuting self-adjoint operators such that $X_k-Y_k\in L_1(H),$ we have
$$\|f(\mathbf{X})-f(\mathbf{Y})\|_{1,\infty}\leq c(d)\|\nabla(f)\|_{\infty}\cdot\max_{1\leq k\leq d}\|X_k-Y_k\|_1.$$
\end{thm}

As a corollary of Theorem \ref{final theorem} we extend our main result from \cite{CPSZ2} to normal operators, see Corollary \ref{Cor=Normal}, which substantially improves corresponding results in \cite{BiSo1}, \cite{CMPS} (see also \cite{APPS}). This extension is based on a strengthened version of the transference principle from \cite{CPSZ2} as explained in Section \ref{Sect=IntSpec}. In the text we prove a somewhat stronger result than Theorem \ref{final theorem} in the terms of double operator integrals (see the next section for the definitions), of which the main Theorem \ref{final theorem} is a corollary.

\begin{thm}\label{doi bound} For every Lipschitz function $f:\mathbb{R}^d\to\mathbb{R}$ and for every collection $\mathbf{A}=\{A_k\}_{k=1}^d$ of commuting self-adjoint operator in a semifinite von Neumann algebra $\mathcal{M},$ we have
$$\|T_{f_{k_0}}^{\mathbf{A},\mathbf{A}}(V)\|_{1,\infty}\leq c(d)\|\nabla(f)\|_{\infty}\|V\|_1,\quad V\in(L_1\cap L_2)(\mathcal{M}),$$
for every $1 \leq k_0 \leq d$.
Here, $f_{k_0}$ is defined by \eqref{fk def}.
\end{thm}

Our proofs are based on weak type versions of de Leeuw theorems \cite{de Leeuw} and a delicate analysis of homogeneous Calder\'on--Zygmund operators.

\section{Preliminaries}

\subsection{General notation}
Throughout the paper $d$ is an integer $\geq 1$. Our main result, Theorem \ref{final theorem}, concerns $d$-tuples of commuting self-adjoint operators, whereas the proofs involve an analysis on $\mathbb{R}^{d+1}$ and $\mathbb{T}^{d+1}$.  We use
\[
\nabla=(\partial_1,\ldots,\partial_{d+1})=\frac{1}{i} (\frac{\partial}{\partial t_1}, \ldots, \frac{\partial}{\partial t_{d+1}})
\]
 for the gradient, which is an unbounded operator on $L_2(\mathbb{R}^{d+1})$. We use $\mathcal{F}$ for the Fourier transform $\mathcal{F}(f)(t) = (2\pi)^{-(d+1)/2}\int_{\mathbb{R}^{d+1}} f(s) e^{-i \langle s, t \rangle} ds$.

Let $\mathcal{M}$ be a semifinite von Neumann algebra equipped with a faithful normal semifinite trace $\tau.$ In this paper, we always presume that $\mathcal{M}$ is represented on a separable Hilbert space.

A (closed and densely defined) operator $x$ affiliated with $\mathcal{M}$ is called $\tau-$measurable if $\tau(E_{|x|}(s,\infty))<\infty$ for sufficiently large $s.$ We denote the set of all $\tau-$measurable operators by $S(\mathcal{M},\tau).$ For every $x\in S(\mathcal{M},\tau),$ we define its singular value function $\mu(A)$ by setting
$$\mu(t,x)=\inf\{\|x(1-p)\|_{\infty}:\quad \tau(p)\leq t\}.$$
Equivalently, for positive self-adjoint operators $x\in S(\mathcal{M},\tau),$ we have
$$n_x(s)=\tau(E_x(s,\infty)),\quad \mu(t,x)=\inf\{s: n_x(s)<t\}.$$
We have for $x,y \in S(\mathcal{M},\tau)$ (see e.g. \cite[Corollary 2.3.16]{LSZ})
\begin{equation}\label{triangle svf}
\mu(t+s,x+y)\leq\mu(t,x)+\mu(s,y),\quad t,s>0.
\end{equation}
Let $S((0,\infty) \times (0, \infty)) = S(L_\infty((0,\infty) \times (0, \infty)), \int\:  ds)$ where the integral is the Lebesgue integral.
Recall that every $x \in S(\mathcal{M},\tau), y\in \mathcal{M}$ such $\mu(x) \otimes \mu(y) \in S((0,\infty) \times (0, \infty))$ we have (see \cite[Eqn. (4.1)]{CPSZ2} for the proof),
\begin{equation}\label{mu tensor}
\mu(x\otimes y)=\mu(\mu(x)\otimes\mu(y)),
\end{equation}
For a measurable function $f$ on $\mathbb{R}^{d+1}$ we use $\sigma_l(f)(t) = f(l^{-1}t), l>0$. Note that
\begin{equation}\label{Eqn=CovariantTrans}
\Vert \sigma_l(f) \Vert_1 = l^{d+1} \Vert f \Vert_1, \quad \Vert \sigma_l(f) \Vert_2 = l^{(d+1)/2} \Vert f \Vert_2,
\end{equation}
where the norms are with respect to the Lebesgue measure on $\mathbb{R}^{d+1}$.

\subsection{Non-commutative spaces}
For $1\leq p<\infty$ we set,
$$L_p(\mathcal{M})=\{x\in S(\mathcal{M},\tau):\ \tau(|x|^p)<\infty\},\quad \|x\|_p=(\tau(|x|^p))^{\frac1p}.$$
The Banach spaces $(L_p(\mathcal{M}),\|\cdot\|_p)$, $1\leq p<\infty$ are separable. Define the space $L_{1,\infty}(\mathcal{M})$ by setting
$$L_{1,\infty}(\mathcal{M})=\{x\in S(\mathcal{M},\tau):\ \sup_{t>0}t\mu(t,x)<\infty\}.$$
We equip $L_{1,\infty}(\mathcal{M})$ with the functional $\|\cdot\|_{1,\infty}$ defined by the formula
$$\|x\|_{1,\infty}=\sup_{t>0}t\mu(t,x),\quad x\in L_{1,\infty}(\mathcal{M}).$$
It follows from \eqref{triangle svf} that
\[
\begin{split}
&\|x+y\|_{1,\infty}=\sup_{t>0}t\mu(t,x+y)\leq\sup_{t>0}t(\mu(\frac{t}{2},x)+\mu(\frac{t}{2},y)) \\
\leq & \sup_{t>0}t\mu(\frac{t}{2},x)+\sup_{t>0}t\mu(\frac{t}{2},y)=2\|x\|_{1,\infty}+2\|y\|_{1,\infty}.
\end{split}
\]
In particular, $\|\cdot\|_{1,\infty}$ is a quasi-norm. The quasi-normed space $(L_{1,\infty}(\mathcal{M}),\|\cdot\|_{1,\infty})$ is, in fact, quasi-Banach (see e.g. \cite[Section 7]{KS} or \cite{Sind}). Naturally we set $L_{1, \infty}(\mathbb{R}^{d+1}) = L_{1, \infty}(L_\infty(\mathbb{R}^{d+1}))$ and $L_{1, \infty}(\mathbb{T}^{d+1}) = L_{1, \infty}(L_\infty(\mathbb{T}^{d+1}))$.


\subsection{Weak type inequalities for Calder\'on-Zygmund operators}
Parcet \cite{Parcet} proved a non-commutative extension of  Calder\'on-Zygmund theory.

Let $K$ be a tempered distribution on $\mathbb{R}^{d+1}$ which we refer to as the {\it convolution kernel}. We let $W_K$ be the associated Calder\'on-Zygmund operator, formally given by $f \mapsto K \ast f.$ In what follows, we only consider tempered distributions having local values (that is, which can be identified with measurable functions $K:\mathbb{R}^{d+1}\rightarrow\mathbb{C}$).

Let $\mathcal{M}$ be a semi-finite von Neumann algebra with normal, semi-finite, faithful trace $\tau.$ The operator $1\otimes W_K$ can, under suitable conditions, be defined as a non-commutative Calder\'on-Zygmund operator by letting it act on the second tensor leg of $L_1(\mathcal{M})\widehat{\otimes}L_1(\mathbb{R}^{d+1}).$ The following theorem in particular gives a sufficient condition for such an operator to act from $L_1$ to $L_{1,\infty}.$ Its proof was improved/shortened very recently by Cadilhac \cite{Cad}.

\begin{thm}[\cite{Cad}, \cite{Parcet}]\label{Thm=Parcet}
Let $K:\mathbb{R}^{d+1} \backslash \{0\} \rightarrow \mathbb{C}$ be a kernel satisfying the conditions
\begin{equation}\label{parcet conditions}
|K|(t)\leq\frac{{\rm const}}{|t|^{d+1}},\quad |\nabla K|(t)\leq\frac{{\rm const}}{|t|^{d+2}}.
\end{equation}
Let $\mathcal{M}$ be a semi-finite von Neumann algebra. If $W_K\in B(L_2(\mathbb{R}^{d+1})),$ then the operator $1\otimes W_K$ defines a bounded map from $L_1(\mathcal{M}\otimes L_\infty(\mathbb{R}^{d+1}))$ to $L_{1,\infty}(\mathcal{M}\otimes L_\infty(\mathbb{R}^{d+1})).$
\end{thm}

We need a very special case of Theorem \ref{Thm=Parcet}.

\begin{thm}\label{parcet corollary} If $g\in L_{\infty}(\mathbb{R}^{d+1})$ is a smooth homogeneous function, then $1\otimes g(\nabla)$ defines a bounded map from $L_1(\mathcal{M}\otimes L_\infty(\mathbb{R}^{d+1}))$ to $L_{1,\infty}(\mathcal{M}\otimes L_\infty(\mathbb{R}^{d+1})).$
\end{thm}
\begin{proof} Without loss of generality, the function $g$ is mean zero on the sphere $\mathbb{S}^d$ (this can be always achieved by subtracting a constant from $g$). By Theorem 6 on p.75 in \cite{Stein} and using that $g$ has mean 0, we have $g(\nabla)=W_K,$ where $K=\mathcal{F}^{-1}(g)$ is a smooth homogeneous function of degree $-d-1$. The gradient of the function $K$ is a smooth homogeneous function of degree $-d-2$. These conditions guarantee that \eqref{parcet conditions} holds for $K$ and by  Theorem \ref{Thm=Parcet}, the assertion follows.
\end{proof}

In Section \ref{Sect=DeLeeuw}, we prove the following compact analogue of Theorem \ref{parcet corollary}. The transference arguments in Section \ref{Sect=IntSpec} require such a compact form.  We let $\nabla_{\mathbb{T}^{d+1}}$ be the gradient operator on the $(d+1)$-torus.

\begin{thm}\label{axis torus} If $g$ is a smooth homogeneous function on $\mathbb{R}^{d+1},$ then the operator $1\otimes g(\nabla_{\mathbb{T}^{d+1}}):L_2(\mathcal{M}\otimes L_{\infty}(\mathbb{T}^{d+1}))\to L_2(\mathcal{M}\otimes L_{\infty}(\mathbb{T}^{d+1}))$ admits a bounded extension acting from $L_1(\mathcal{M}\otimes L_{\infty}(\mathbb{T}^{d+1}))$ to $L_{1,\infty}(\mathcal{M}\otimes L_{\infty}(\mathbb{T}^{d+1})).$
\end{thm}

\begin{rmk}\label{Rmk=DeLeeuw}
Theorem \ref{axis torus} should be understood as a de Leeuw theorem in the following sense. Assume for simplicity that $\mathcal{M} = \mathbb{C}$. $g(\nabla)$ of Theorem \ref{parcet corollary} is a Fourier multiplier with symbol $g$. $g(\nabla_{\mathbb{T}^{d+1}})$ is the Fourier multiplier on $L_2(\mathbb{T}^{d+1})$ whose symbol is the restriction of $g$ to $\mathbb{Z}^{d+1}$.  Theorem \ref{axis torus} then shows that $g\vert_{\mathbb{Z}^{d+1}}$ is the symbol of a bounded multiplier $L_1(\mathbb{T}^{d+1}) \rightarrow L_{1, \infty}(\mathbb{T}^{d+1})$. This is a weak $(1,1)$ version of de Leeuw's theorem \cite{de Leeuw}.
\end{rmk}

\subsection{Double operator integrals}\label{doi subsection} Let $\mathbf{A}=\{A_k\}_{k=1}^d$ be a collection of {\it commuting} self-adjoint operators affiliated with $\mathcal{M}.$ Consider projection valued measures on $\mathbb{R}^d$ acting on the Hilbert space $L_2(\mathcal{M})$ by the formulae
$$x\to\Big(\prod_{k=1}^dE_{A_k}(\mathcal{B}_k)\Big)x,\quad x\to x\Big(\prod_{k=1}^dE_{A_k}(\mathcal{C}_k)\Big),\quad x\in L_2(\mathcal{M}).$$
These spectral measures commute and, hence (see Theorem V.2.6 in \cite{BirSol}), there exists a countably additive (in the strong operator topology) projection-valued measure $\nu$ on $\mathbb{R}^2$ acting on the Hilbert space $L_2(\mathcal{M})$ by the formula
\begin{equation}\label{birsol spectral measure}
\nu(\mathcal{B}_1\times\cdots\times\mathcal{B}_d\times\mathcal{C}_1\times\cdots\times\mathcal{C}_d):x\to \Big(\prod_{k=1}^dE_{A_k}(\mathcal{B}_k)\Big)x\Big(\prod_{k=1}^dE_{A_k}(\mathcal{C}_k)\Big),\quad x\in L_2(\mathcal{M}).
\end{equation}
Integrating a bounded Borel function $\xi$ on $\mathbb{R}^{2d}$ with respect to the measure $\nu$ produces a bounded operator acting on the Hilbert space $L_2(\mathcal{M}).$ In what follows, we denote the latter operator by $T_{\xi}^{\mathbf{A},\mathbf{A}}$ (see also \cite[Remark 3.1]{PSW}).

In the special case when $A_k$ are bounded and ${\rm spec}(A_k)\subset\mathbb{Z},$ we have
\begin{equation}\label{doi special}
T_{\xi}^{\mathbf{A},\mathbf{A}}(V)=\sum_{\mathbf{i},\mathbf{j}\in\mathbb{Z}^d}\xi(\mathbf{i},\mathbf{j})\Big(\prod_{k=1}^dE_{A_k}(\{i_k\})\Big)V\Big(\prod_{k=1}^dE_{A_k}(\{j_k\})\Big).
\end{equation}
We are mostly interested in the case $\xi=f_k$ for a Lipschitz function $f.$ Here, for $1 \leq k \leq d$ and $\lambda, \mu \in \mathbb{R}^{d}$,
\begin{equation}\label{fk def}
f_k(\lambda,\mu)=
\begin{cases}
\frac{(f(\lambda)-f(\mu))(\lambda_k-\mu_k)}{\langle\lambda-\mu,\lambda-\mu\rangle},\quad \lambda\neq\mu\\
0,\quad  \lambda=\mu.
\end{cases}
\end{equation}

\section{A de Leeuw type theorem for Calder\'on-Zygmund operators}\label{Sect=DeLeeuw}

In this section we collect de Leeuw type results (c.f. \cite{de Leeuw}) needed in the subsequent proofs. The main result is Theorem  \ref{axis torus}. This theorem should be understood as a restriction theorem for (homogeneous) Fourier multipliers, see Remark \ref{Rmk=DeLeeuw}.

The strategy of the proof is as follows. One  finds an asymptotic embedding of $L_1(\mathbb{T}^{d+1})$ (resp. $L_{1, \infty}(\mathbb{T}^{d+1})$) into $L_1(\mathbb{R}^{d+1})$ (resp. $L_{1, \infty}(\mathbb{R}^{d+1})$) such that this asymptotic embedding intertwines the Fourier multipliers/Calder\'on-Zygmund operators and their discretizations.

\vspace{0.3cm}

In what follows,
$$G_l(t)=  (l \sqrt{2\pi})^{-(d+1)} e^{-\frac{|t|^2}{2l^2}},\quad t\in\mathbb{R}^{d+1}, \quad l>0.$$
We have that $\|G_l\|_1=1$.
Let $\mathcal{F}$ stand for the Fourier transform.
Note that
\begin{equation}\label{Eqn=Covariance}
\begin{split}
 (\mathcal{F}G_l)(t)= &
(l \sqrt{2\pi})^{-(d+1)}\int e^{ - \frac{ \vert s \vert^2}{2 l^2} }  e^{-i \langle t,s \rangle}  ds \\
=& (l \sqrt{2\pi})^{-(d+1)}   \int e^{ - \frac{ \vert s \vert^2}{2} }  e^{-i \langle l t,s \rangle}  ds
= G_1(lt).
\end{split}
\end{equation}
 We set
\begin{equation}\label{ek def}
e_k(t):=e^{i\langle k,t\rangle},\quad k,t\in\mathbb{R}^{d+1}, k \in \mathbb{Z}^{d+1}.
\end{equation}
Let $\alpha = (\alpha_1, \ldots, \alpha_{d+1})$ where $\alpha_k \in \mathbb{Z}_+$. The notation $\partial^{\alpha}$ is used for $h(\nabla),$ where $h(t)=\prod_{k=1}^{d+1}t_k^{\alpha_k},$ $t\in\mathbb{R}^{d+1}.$ We have
\[
M_{e_{-k}} g(\nabla) M_{e_k} =  g(M_{e_{-k}} \nabla  M_{e_k}) = g(\nabla + k).
\]

\begin{rmk}
The Gaussian functions $G_l$ are needed to normalize our asymptotic embeddings given by periodizations of functions (see Lemmas \ref{first periodic lemma} and \ref{second periodic lemma} for exact statements). These asymptotic embeddings are closely related to the   Bohr compactification of $\mathbb{R}^{d+1}$.
\end{rmk}

The following lemma is a $(d+1)-$dimensional analogue of Lemma 7 in \cite{PScrelle}.

\begin{lem}\label{fourier estimate lemma} For every function $h$ on $\mathbb{R}^{d+1}$ whose partial derivatives up to order $d+1$ belong to $L_2(\mathbb{R}^{d+1})$ we have
$$\|\mathcal{F}^{-1}(h)\|_1\leq 2^{\frac{d+1}{2}}\sum_{|\alpha|\leq d+1}\|\partial^{\alpha}(h)\|_2.$$
\end{lem}
\begin{proof} For every $\mathscr{A}\subset\{1,\cdots,d+1\},$ we define the set $O_{\mathscr{A}}\subset\mathbb{R}^{d+1}$ by setting
$$O_{\mathscr{A}}=\{t\in\mathbb{R}^{d+1}:\ |t_k|\geq 1,\ k\in\mathscr{A},\ |t_k|\leq 1,\ k\notin\mathscr{A}\}.$$
We also define the function $h_{\mathscr{A}}$ on $\mathbb{R}^{d+1}$ by setting
$$h_{\mathscr{A}}(t)=\prod_{k\in\mathscr{A}}t_k,\quad t\in\mathbb{R}^{d+1}.$$
Note that the sets $O_{\mathscr{A}}$ form a partition of $\mathbb{R}^{d+1}$ and that   for every choice of $\mathcal{A}$ we have $\Vert h_{\mathscr{A}}^{-1} \chi_{O_{\mathscr{A}}} \Vert_2 \leq 2^{\frac{d+1}{2}}$.

We have
\[
\|\mathcal{F}^{-1}(h)\|_1\leq\sum_{\mathscr{A}\subset\{1,\cdots,d+1\}}\|\mathcal{F}^{-1}(h)\chi_{O_{\mathscr{A}}}\|_1.
\]
By the H\"older inequality
\[
\|\mathcal{F}^{-1}(h)\|_1\leq \sum_{\mathscr{A}\subset\{1,\cdots,d+1\}}\|h_{\mathscr{A}}\mathcal{F}^{-1}(h)\chi_{O_{\mathscr{A}}}\|_2\|h_{\mathscr{A}}^{-1}\chi_{O_{\mathscr{A}}}\|_2.
\]
By the previous paragraph and the Plancherel identity
\[
\|\mathcal{F}^{-1}(h)\|_1\leq 2^{\frac{d+1}{2}}\sum_{\mathscr{A}\subset\{1,\cdots,d+1\}}\|\mathcal{F}^{-1}(h_{\mathscr{A}}(\nabla)h)\|_2=2^{\frac{d+1}{2}}\sum_{\mathscr{A}\subset\{1,\cdots,d+1\}}\|h_{\mathscr{A}}(\nabla)h\|_2.
\]
 The proof follows as $h_{\mathscr{A}}(\nabla) = \partial^{\alpha}$.
\end{proof}

For a multi-index $\alpha = (\alpha_1, \ldots, \alpha_{d+1}) \in \mathbb{Z}_+^{d+1}$ let $\vert \alpha \vert = \sum_{i=1}^{d+1} \alpha_i$. We shall without further reference use the fact that $\partial^\alpha(\sigma_l(f)) = l^{- \vert \alpha \vert} \sigma_l(\partial^\alpha(f))$ for any smooth function $f$ on $\mathbb{R}^{d+1}$.

\begin{lem}\label{first convergence lemma} Let $g\in L_{\infty}(\mathbb{R}^{d+1})$ be a smooth function with all derivatives assumed to be uniformly bounded. If $(\partial^{\alpha}g)(0)=0$ for every multi-index $\alpha$ with $|\alpha|\leq d,$ then
$$\|(g(\nabla))(G_l)\|_1\to0,\quad l\to\infty.$$
\end{lem}
\begin{proof} We have $g(\nabla)=\mathcal{F}^{-1}M_g\mathcal{F}$, with $M_g$ the multiplication operator with $g$ on $L_2(\mathbb{R}^{d+1})$.
Recall again that $\mathcal{F}(G_l)(t) = G_1(lt), t\in\mathbb{R}^{d+1}.$
Thus, see e.g. \eqref{Eqn=Covariance},
$$(g(\nabla))(G_l)=\mathcal{F}^{-1}M_g\mathcal{F}(G_l)=\mathcal{F}^{-1}(gh_l),$$
where $h_l(t)=G_1(lt),$ $t\in\mathbb{R}^{d+1}.$ It follows from Lemma \ref{fourier estimate lemma} that
$$\|\mathcal{F}^{-1}(gh_l)\|_1\leq 2^{\frac{d+1}{2}}\sum_{|\alpha|\leq d+1}\|\partial^{\alpha}(gh_l)\|_2\leq 2^{\frac{d+1}{2}}\sum_{|\alpha|+|\beta|\leq d+1}\|\partial^{\alpha}(g)\partial^{\beta}(h_l)\|_2.$$
Due to the assumption that $(\partial^{\alpha}g)(0)=0$ for every multi-index $\alpha$ with $|\alpha|\leq d$, all coefficients in the Taylor expansion of $g$ around 0 of the terms of order $\leq d$ vanish. Therefore, as all derivatives of $g$ are assumed to be uniformly bounded functions we obtain that $|\partial^{\alpha}g|\leq c(g)f^{d+1-|\alpha|},$ $|\alpha|\leq d+1,$ where $f(t)=\vert t \vert$, for some constant $c(g)$.   Thus,
$$\|\mathcal{F}^{-1}(gh_l)\|_1\leq 2^{\frac{d+1}{2}}c(g)\sum_{|\alpha|+|\beta|\leq d+1}\|f^{d+1-|\alpha|}\partial^{\beta}(h_l)\|_2.$$
We have
$$\partial^{\beta}(h_l)=l^{|\beta|}\sigma_{\frac1l}(\partial^{\beta}G_1),\quad f^{d+1-|\alpha|}=l^{|\alpha|-d-1}\sigma_{\frac1l}(f^{d+1-|\alpha|}).$$
Thus,
\[
\begin{split}
& \|f^{d+1-|\alpha|}\partial^{\beta}(h_l)\|_2
=  l^{\vert \beta\vert+\vert\alpha\vert-d-1}\|\sigma_{\frac1l}(f^{d+1-|\alpha|}\partial^{\beta}(G_1))\|_2=\\
= & l^{  \vert \beta \vert + \vert \alpha \vert - \frac{3}{2}(d+1)  }   \|f^{d+1-|\alpha|}\partial^{\beta}G_1\|_2 \rightarrow 0.
\end{split}
\]
This concludes the proof.
\end{proof}

\begin{lem}\label{second convergence lemma} If $g:\mathbb{R}^{d+1}\to\mathbb{C}$ is a Schwartz function such that $g(0)=0,$ then
$$\|(g(\nabla))(G_l)\|_1\to0,\quad l\to\infty.$$
\end{lem}
\begin{proof} Define Schwartz functions $g_j:\mathbb{R}^{d+1}\to\mathbb{C},$ $1\leq j\leq d+1,$ by setting
$$g_j(t)=\frac{g(0,\cdots,0,t_j,\cdots,t_{d+1})-g(0,\cdots,0,t_{j+1},\cdots,t_{d+1})}{t_j},\quad t\in\mathbb{R}^{d+1}.$$
We have,
$$g(t)= \sum_{j=1}^{d+1}t_j g_j(t).$$
and, therefore,
\begin{equation}\label{scl1}
g(\nabla)(G_l)=\sum_{j=1}^{d+1} g_j(\nabla) \cdot \Big( \partial_j G_l \Big).
\end{equation}
It follows from Young inequality that
$$\|g_j(\nabla)x\|_1=\|\mathcal{F}^{-1}M_{g_j}\mathcal{F}x\|_1=\|\mathcal{F}^{-1}(g_j)\ast x\|_1\leq\|\mathcal{F}^{-1}(g_j)\|_1\|x\|_1,\quad x\in L_1(\mathbb{R}^{d+1}).$$
The proof then follows provided that for $x = \partial_j G_l, 1 \leq j \leq d+1$ we have,
\begin{equation}\label{Eqn=WriteOut}
\|\partial_jG_l\|_1\to0,\quad l\to\infty.
\end{equation}
Indeed, a direct computation yields,
$$\partial_jG_l=\frac1{l^{d+2}}\sigma_l(h_j),\quad \textrm{ where } h_j(t) :=it_jG_1(t),\quad t\in\mathbb{R}^{d+1}.$$
So appealing to \eqref{Eqn=CovariantTrans}, we obtain
\[
\| \partial_j(G_l )\|_1 = \frac{1}{l}  \Vert h_j\Vert_1 \rightarrow 0.
\]
\end{proof}

\begin{lem}\label{third convergence lemma} Let $g\in L_{\infty}(\mathbb{R}^{d+1})$ be a smooth function with all its derivatives assumed to be uniformly bounded. If $k\in\mathbb{R}^{d+1},$ then
$$\|(g(\nabla))(G_le_k)-g(k)G_le_k\|_1\to0,\quad l\to\infty.$$
Here $e_k$ is given by \eqref{ek def}.
\end{lem}
\begin{proof} Suppose first that $k=0$ and $g(0)=0.$ Let $\psi$ be a Schwartz function on $\mathbb{R}^{d+1}$ such that $\psi(t)=1$ whenever $|t|\leq 1.$ Set
$$\phi(t)=\sum_{|\alpha|\leq d}\frac{i^{|\alpha|}}{\prod_{k=1}^{d+1}(\alpha_k)!}(\partial^{\alpha}g)(0)t^{\alpha}\psi(t),\quad t\in\mathbb{R}^{d+1}.$$
Clearly, $\phi$ is a Schwartz function, $\phi(0) = 0$ and $(\partial^{\alpha}g)(0)=(\partial^{\alpha}\phi)(0)$ for $|\alpha|\leq d.$ In other words, the function $g-\phi$ satisfies the assumptions of Lemma \ref{first convergence lemma}. Using Lemmas \ref{first convergence lemma} and \ref{second convergence lemma}, we obtain
$$\|((g-\phi)(\nabla))(G_l)\|_1\to0,\quad \|(\phi(\nabla))(G_l)\|_1\to0,\quad l\to\infty.$$
Using triangle inequality, we obtain
$$\|(g(\nabla))(G_l)\|_1\to0,\quad l\to\infty.$$
This proves the assertion in our special case.

To prove the assertion in general, note that
\begin{equation}\label{Eqn=Limit}
\begin{split}
& \Vert g(\nabla) (G_l e_k) - g(k) G_l e_k \Vert_1
= \Vert (M_{e_{-k}} g(\nabla) M_{e_k} - g(k)) (G_l ) \Vert_1 \\
= & \Vert  ( g(\nabla +k) - g(k)  )(G_l) \Vert_1.
\end{split}
\end{equation}
Now as  $t\to g(t+k)-g(k)$ is a function satisfying the assumptions of the first paragraph, we see that \eqref{Eqn=Limit} goes to 0 as $l \rightarrow \infty$.
\end{proof}

The following Lemma \ref{fourth convergence lemma} is the main intertwining property as we explained in the beginning of this section.

\begin{lem}\label{fourth convergence lemma} Let $g\in L_{\infty}(\mathbb{R}^{d+1})$ be a smooth (except at $0$) homogeneous function of degree $0.$ For every $0\neq k\in\mathbb{R}^{d+1},$ we have
$$\|(g(\nabla))(G_le_k)-g(k)G_le_k\|_{1,\infty}\to0,\quad l\to\infty.$$
\end{lem}
\begin{proof} Fix $0\neq k\in\mathbb{R}^{d+1}.$ Fix a Schwartz function $\phi$ supported on the ball $\{|t|_2 < \vert k \vert_2 \}$ such that $\phi(t)=1$ whenever $|t|_2\leq \frac12|k|_2.$ Clearly, both functions $\phi$ and $g(1-\phi)$ satisfy the conditions of Lemma \ref{third convergence lemma}. We obtain
\[
\begin{split}
&\|((g(1-\phi))(\nabla))(G_le_k)-g(k)G_le_k\|_{1,\infty} \\
\leq &
\|((g(1-\phi))(\nabla))(G_le_k)-g(k)G_le_k\|_{1} \to0,\quad l\to\infty
\end{split}
\]
And also
\[
\|(\phi(\nabla))(G_le_k) \|_{1}\to0,\quad l\to\infty.
\]

By Theorem 1 on p.29 in \cite{Stein} (see especially Step 2 on p.30; one can also use Theorem \ref{parcet corollary} here), the operator  $g(\nabla):L_1(\mathbb{R}^{d+1})\to L_{1,\infty}(\mathbb{R}^{d+1})$ is bounded. Thus, since $\phi$ satisfies the assumptions of Lemma \ref{first convergence lemma},
$$\|(g\phi(\nabla))(G_le_k)\|_{1,\infty}\leq\|g(\nabla)\|_{L_1\to L_{1,\infty}}\|(\phi(\nabla))(G_le_k)\|_1\to0,\quad l\to\infty.$$
The assertion follows by applying triangle inequality.
\end{proof}

\begin{lem}\label{tensor mu} Let $A\in L_1(\mathcal{M}_1)$ and let $B\in L_{1,\infty}(\mathcal{M}_2).$ We have
$$\|A\otimes B\|_{1,\infty}\leq\|A\|_1\|B\|_{1,\infty}.$$
\end{lem}
\begin{proof} Define the function $z$ on $(0,\infty)$ by setting $z(t):=t^{-1},$ $t>0.$ We have
$$\mu(A\otimes B)\stackrel{\eqref{mu tensor}}{=}\mu(\mu(A)\otimes\mu(B))\leq\|B\|_{1,\infty}\mu(\mu(A)\otimes z).$$

We claim that for every positive decreasing function $x\in L_1(0,\infty),$ we have $\mu(x\otimes z)=\|x\|_1z.$ Set $x_n=\sum_{k=0}^{n^2-1}\mu(\frac{k+1}{n},x)\chi_{(\frac{k}{n},\frac{k+1}{n})}, n > 1.$ The functions $\chi_{(\frac{k}{n},\frac{k+1}{n})}\otimes z,$ $0\leq k<n^2,$ are disjointly supported and equimeasurable with $\frac1nz.$ Therefore,
$$\mu(x_n\otimes z)=\mu(\sum_{k=0}^{n^2-1}\mu(\frac{k+1}{n},x)\chi_{(\frac{k}{n},\frac{k+1}{n})}\otimes z)=\mu(\bigoplus_{k=0}^{n^2-1}\frac1n\mu(\frac{k+1}{n},x)z)=\|x_n\|_1z.$$
It is immediate that $x_n\uparrow x$ and, therefore, $x_n\otimes z\uparrow x\otimes z$ and $\mu(x_n\otimes z)\uparrow \mu(x\otimes z).$ This proves the claim.
\end{proof}

Let
\[
{\rm per}:\mathcal{M}\otimes L_{\infty}(\mathbb{T}^{d+1})\to \mathcal{M}\otimes L_{\infty}(\mathbb{R}^{d+1})
\]
 be the natural embedding by periodicity. Under the identification $\mathcal{M}\otimes L_{\infty}(\mathbb{R}^{d+1}) \simeq L_\infty(\mathbb{R}^{d+1}, \mathcal{M})$ (the latter being understood as weakly measurable, essentially bounded functions) and similarly for the torus, it is defined as
 \[
 {\rm per}(f)(t) = f(t \:\: {\rm mod} \: 2\pi), \qquad t \in \mathbb{R}^{d+1}.
 \]
We consider $\mathbb{T}$ with total Haar measure $2 \pi$. The next Lemma \ref{first periodic lemma} provides the asymptotic embedding of $L_1(\mathbb{T}^{d+1})$ to  $L_1(\mathbb{R}^{d+1})$.

\begin{lem}\label{first periodic lemma} For every $W\in L_1(\mathcal{M}\otimes L_{\infty}(\mathbb{T}^{d+1})),$ we have
$$\lim_{l\to\infty}\|{\rm per}(W)\cdot(1\otimes G_l)\|_{L_1(\mathcal{M}\otimes L_{\infty}(\mathbb{R}^{d+1}))}=\frac1{(2\pi)^{d+1}}\|W\|_{L_1(\mathcal{M}\otimes L_{\infty}(\mathbb{T}^{d+1}))}.$$
\end{lem}
\begin{proof}  For every $m\in\mathbb{Z},$ define $l(m),n(m)\in\mathbb{Z}$ by setting
\begin{equation}\label{Eqn=LDef}
l(m)=
\begin{cases}
m&m\geq0\\
m+1&m<0
\end{cases},
\end{equation}
\begin{equation}\label{Eqn=NDef}
n(m)=
\begin{cases}
m+1&m\geq0\\
m&m<0
\end{cases}.
\end{equation}
Next set
\[
\begin{split}
l(m) = (l(m_1), \ldots, l(m_{d+1})),& \qquad m \in \mathbb{Z}^{d+1}, \\
n(m) =  (n(m_1), \ldots, n(m_{d+1})),& \qquad m \in \mathbb{Z}^{d+1}.
\end{split}
\]
Clearly,
\[
\begin{split}
&  \|{\rm per}(W)\cdot(1\otimes G_l)\|_{L_1(\mathcal{M}\otimes L_{\infty}(\mathbb{R}^{d+1}))} \\
= &  \sum_{m\in\mathbb{Z}^{d+1}}\|{\rm per}(W)\cdot(1\otimes G_l)\cdot(1\otimes\chi_{2\pi m+[0,2\pi]^{d+1}})\|_{L_1(\mathcal{M}\otimes L_{\infty}(\mathbb{R}^{d+1}))}.
\end{split}
\]
By construction,
\[
G_l(2\pi n(m))\leq G_l(t)\leq G_l(2\pi l(m)),\quad t\in 2\pi m+[0,2\pi]^{d+1}.
\]
Hence,
\[
\begin{split}
& \|{\rm per}(W)\cdot(1\otimes G_l)\|_{L_1(\mathcal{M}\otimes L_{\infty}(\mathbb{R}^{d+1}))} \\
\leq & \sum_{m\in\mathbb{Z}^{d+1}}G_l(2\pi l(m))\|{\rm per}(W)\cdot(1\otimes\chi_{2\pi m+[0,2\pi]^{d+1}})\|_{L_1(\mathcal{M}\otimes L_{\infty}(\mathbb{R}^{d+1}))} \\
= &\|W\|_{L_1(\mathcal{M}\otimes L_{\infty}(\mathbb{T}^{d+1}))}\cdot \sum_{m\in\mathbb{Z}^{d+1}}G_l(2\pi l(m)).
\end{split}
\]
Similarly,
\[
\begin{split}
& \|{\rm per}(W)\cdot(1\otimes G_l)\|_{L_1(\mathcal{M}\otimes L_{\infty}(\mathbb{R}^{d+1}))} \\
\geq &\sum_{m\in\mathbb{Z}^{d+1}}G_l(2\pi n(m))\|{\rm per}(W)\cdot(1\otimes\chi_{2\pi m+[0,2\pi]^{d+1}})\|_{L_1(\mathcal{M}\otimes L_{\infty}(\mathbb{R}^{d+1}))} \\
= &\|W\|_{L_1(\mathcal{M}\otimes L_{\infty}(\mathbb{T}^{d+1}))}\cdot \sum_{m\in\mathbb{Z}^{d+1}}G_l(2\pi n(m)).
\end{split}
\]
We have
\[
\begin{split}
&\sum_{m\in\mathbb{Z}^{d+1}}G_l(2\pi l(m))
= \left(  \sum_{m\in\mathbb{Z}} G_l(2\pi l(m))
   \right)^{d+1} \\
= &
\Big(\frac1{l\sqrt{2\pi}}+\frac1{l\sqrt{2\pi}}\sum_{m\in\mathbb{Z}}e^{-\frac{(2\pi m)^2}{2l^2}}\Big)^{d+1}\to\frac1{(2\pi)^{d+1}},\quad l\to\infty,
\end{split}
\]
where the limit is by elementary  Riemann integration. Similarly
\[
\begin{split}
&\sum_{m\in\mathbb{Z}^{d+1}}G_l(2\pi n(m))
= \left(  \sum_{m\in\mathbb{Z}} G_l(2\pi n(m))
   \right)^{d+1} \\
= &
\Big( -\frac1{l\sqrt{2\pi}}+\frac1{l\sqrt{2\pi}}\sum_{m\in\mathbb{Z}}e^{-\frac{(2\pi m)^2}{2l^2}}\Big)^{d+1}\to\frac1{(2\pi)^{d+1}},\quad l\to\infty.
\end{split}
\]
Combining the last 4 equations  completes the proof as they show that we have estimates
\[
\begin{split}
& \frac1{(2\pi)^{d+1}}\|W\|_{L_1(\mathcal{M}\otimes L_{\infty}(\mathbb{T}^{d+1}))} - \epsilon_l  \\
 \leq &  \|{\rm per}(W)\cdot(1\otimes G_l)\|_{L_1(\mathcal{M}\otimes L_{\infty}(\mathbb{R}^{d+1}))} \leq \frac1{(2\pi)^{d+1}}\|W\|_{L_1(\mathcal{M}\otimes L_{\infty}(\mathbb{T}^{d+1}))} + \epsilon_l.
\end{split}
\]
for some sequences $\epsilon_l >0$ that converges to 0.
\end{proof}

The next lemma gives the asymptotic norm estimate  of periodizations of  elements of $L_{1, \infty}(\mathbb{T}^{d+1})$ with the norms of  $L_{1, \infty}(\mathbb{R}^{d+1})$.

\begin{lem}\label{second periodic lemma}
For every $W\in L_{1,\infty}(\mathcal{M}\otimes L_{\infty}(\mathbb{T}^{d+1})),$ we have
$$\liminf_{l\to\infty}\|{\rm per}(W)\cdot(1\otimes G_l)\|_{L_{1,\infty}(\mathcal{M}\otimes L_{\infty}(\mathbb{R}^{d+1}))}\gtrsim\|W\|_{L_{1,\infty}(\mathcal{M}\otimes L_{\infty}(\mathbb{T}^{d+1}))}.$$
Here, $\gtrsim$ means inequality up to some constant independent of $W$.
\end{lem}
\begin{proof} We estimate crudely,
\[
\begin{split}
G_l(t)\geq & c(d)l^{-d-1},\quad |t|\leq 4\pi l,\\
\chi_{\{|t|\leq 4\pi l\}}\geq & \sum_{|m|\leq l}\chi_{2\pi m+[0,2\pi]^d}.
\end{split}
\]
Hence,
\[
\begin{split}
     &  \|{\rm per}(W)\cdot(1\otimes G_l)\|_{L_{1,\infty}(\mathcal{M}\otimes L_{\infty}(\mathbb{R}^{d+1}))} \\
\geq &  c(d)l^{-d-1}\|{\rm per}(W)\cdot(1\otimes \sum_{|m|\leq l}\chi_{2\pi m+[0,2\pi]^d})\|_{L_{1,\infty}(\mathcal{M}\otimes L_{\infty}(\mathbb{R}^{d+1}))}.
\end{split}
\]
Since the elements ${\rm per}(W)\cdot(1\otimes  \chi_{2\pi m+[0,2\pi]^d})$ with $|m|\leq l$ are pairwise orthogonal we have that
$${\rm per}(W)\cdot(1\otimes \sum_{|m|\leq l}\chi_{2\pi m+[0,2\pi]^d}) \in L_{1,\infty}(\mathcal{M}\otimes L_{\infty}(\mathbb{R}^{d+1})  )$$
and
$$\bigoplus_{|m|\leq l} W   \in L_{1,\infty}(\mathcal{M}\otimes L_{\infty}(\mathbb{T}^{d+1})\otimes l_{\infty})$$
are unitarily equivalent. Then
\[
\begin{split}
& \|{\rm per}(W)\cdot(1\otimes G_l)\|_{L_{1,\infty}(\mathcal{M}\otimes L_{\infty}(\mathbb{R}^{d+1}))}
 \geq   c(d)l^{-d-1}\|\bigoplus_{|m|\leq l}W\|_{L_{1,\infty}(\mathcal{M}\otimes L_{\infty}(\mathbb{T}^{d+1})\otimes l_{\infty})}.
 \end{split}
\]
Let $n_l$ be the number of $m \in \mathbb{Z}^{d+1}$ with $\vert m \vert_2 \leq l$. Note that $n_l \gtrsim l^{d+1}$.  Then $\mu(t, \bigoplus_{|m|\leq l}W ) = \mu( n_l^{-1} t,  W )$ from which we may continue the estimate
\[
\begin{split}
& \|{\rm per}(W)\cdot(1\otimes G_l)\|_{L_{1,\infty}(\mathcal{M}\otimes L_{\infty}(\mathbb{R}^{d+1}))}
\geq    c(d)l^{-d-1}  n_l \|W\|_{L_{1,\infty}(\mathcal{M}\otimes L_{\infty}(\mathbb{T}^{d+1}))} \\
\geq & c(d) \|W\|_{L_{1,\infty}(\mathcal{M}\otimes L_{\infty}(\mathbb{T}^{d+1}))}.
\end{split}
\]
\end{proof}

We are now fully equipped to prove our main result.

\begin{proof}[Proof of Theorem \ref{axis torus}] Let $\mathscr{A}\subset\mathbb{Z}^{d+1}$ be a finite set. Let
$$W=\sum_{k\in\mathscr{A}}W_k\otimes e_k,\quad W_k\in L_1(\mathcal{M}).$$
Firstly, we prove
\[
\Vert (1 \otimes g(\nabla))(W)\Vert_{1, \infty} \lesssim \Vert W \Vert_1,
\]
 for $W$ as above. As conditional expectations are contractions on $L_1$ we have
\[
\begin{split}
&  \Vert \sum_{0  \not = k \in \mathscr{A} }  W_k \otimes e_k\Vert_1
\leq
\|\sum_{k\in\mathscr{A}} W_k\otimes e_k\|_1 + \Vert W_0 \otimes e_0 \Vert_1
\leq   2\|W\|_1, \qquad k \in \mathcal{A}.
\end{split}
\]
Therefore, we may (and will) assume without loss of generality that $0\notin\mathscr{A}.$ By Theorem \ref{Thm=Parcet}, we have
$$\|(1\otimes g(\nabla))({\rm per}(W)\cdot (1\otimes G_l))\|_{L_{1,\infty}(\mathcal{M}\otimes L_{\infty}(\mathbb{R}^{d+1}))}\leq \|{\rm per}(W)\cdot (1\otimes G_l)\|_{L_1(\mathcal{M}\otimes L_{\infty}(\mathbb{R}^{d+1}))}. $$
By respectively Lemma \ref{tensor mu} and Lemma   \ref{fourth convergence lemma}   we have for each $k \in \mathcal{A}$ as $l \rightarrow \infty$,
\[
\begin{split}
&  \Vert (1 \otimes g(\nabla))(W_k \otimes G_l e_k) - g(k) (W_k \otimes G(l) e_k) \Vert_{1, \infty} \\
\leq  & \Vert W_k \Vert_1 \Vert  (g(\nabla))(G_l e_k)  - g(k) G_l e_k  \Vert_{1, \infty} \rightarrow 0.
\end{split}
\]
The quasi-triangle inequality gives for  sums of arbitrary operators $x_\alpha$  that
$$\Vert \sum_{\alpha \in \mathcal{A}} x_\alpha \Vert_{1, \infty} \leq 2^{\vert \mathcal{A} \vert} \sum_{\alpha \in \mathcal{A}} \Vert x_\alpha \Vert_{1, \infty}.$$
So it follows that as $l \rightarrow \infty$
\[
\Vert \sum_{\alpha \in \mathcal{A}} (1 \otimes g(\nabla))(W_k \otimes G_l e_k) - \sum_{\alpha \in \mathcal{A}}  g(k) (W_k \otimes G(l) e_k) \Vert_{1, \infty} \rightarrow 0.
\]
In other words we have   as $l \rightarrow \infty$
\[
\Vert (1 \otimes g(\nabla)) ( {\rm per}(W) \cdot (1 \otimes G_l)   ) - {\rm per} (  (1 \otimes g(\nabla) ) (W) ) (1 \otimes G_l) \Vert_{1, \infty} \rightarrow 0.
\]
  Thus,
\begin{equation}\label{Eqn=SomeLimit}
\begin{split}
& \liminf_{l\to\infty}\|{\rm per}((1\otimes g(\nabla))(W))\cdot (1\otimes G_l)\|_{L_{1,\infty}(\mathcal{M}\otimes L_{\infty}(\mathbb{R}^{d+1}))}  \\
\leq & \liminf_{l\to\infty}\|{\rm per}(W)\cdot (1\otimes G_l)\|_{L_1(\mathcal{M}\otimes L_{\infty}(\mathbb{R}^{d+1}))}.
\end{split}
\end{equation}
It follows now from Lemma \ref{second periodic lemma},  \eqref{Eqn=SomeLimit} and Lemma \ref{first periodic lemma}   that
\begin{equation}\label{estimate specific}
\begin{split}
& \|(1\otimes g(\nabla_{\mathbb{T}^{d+1}}))(W)\|_{L_{1,\infty}(\mathcal{M}\otimes L_{\infty}(\mathbb{T}^{d+1}))}\\
\lesssim &
\liminf_{l\to\infty}\|{\rm per}((1\otimes g(\nabla_{\mathbb{T}^{d+1}}))W)\cdot(1\otimes G_l)\|_{L_{1,\infty}(\mathcal{M}\otimes L_{\infty}(\mathbb{R}^{d+1}))} \\
 \leq & \liminf_{l\to\infty}\|{\rm per}(W)\cdot (1\otimes G_l)\|_{L_1(\mathcal{M}\otimes L_{\infty}(\mathbb{R}^{d+1}))} \\
 \lesssim &
\|W\|_{L_1(\mathcal{M}\otimes L_{\infty}(\mathbb{T}^{d+1}))}.
\end{split}
\end{equation}
This proves the assertion for our specific $W$.

To see the assertion in general, fix an arbitrary $W\in L_1(\mathcal{M}\otimes L_{\infty}(\mathbb{T}^{d+1}))$ and choose $W^m$ as above such that $W^m\to W$ in $L_1(\mathcal{M}\otimes L_{\infty}(\mathbb{T}^{d+1}))$ as $m\to\infty$ (see Lemma \ref{Lem=Fejer}). In particular, the sequence $\{W^m\}_{m\geq1}\subset L_1(\mathcal{M}\otimes L_{\infty}(\mathbb{T}^{d+1}))$ is Cauchy. By \eqref{estimate specific}, the sequence $\{(1\otimes g(\nabla))(W^m)\}_{m\geq1}\subset L_{1,\infty}(\mathcal{M}\otimes L_{\infty}(\mathbb{T}^{d+1}))$ is also Cauchy. Denote the limit by $T(W).$ If also $W\in L_2(\mathcal{M}\otimes L_{\infty}(\mathbb{T}^{d+1})),$ then the sequence $\{W^m\}_{m\geq1}$ can be chosen such that also $W^m\to W$ in $L_2(\mathcal{M}\otimes L_{\infty}(\mathbb{T}^{d+1}))$ (see Remark \ref{Rmk=Fejer2}). Thus, $T(W)=(1\otimes g(\nabla))(W)$ for $W\in(L_1\cap L_2)(\mathcal{M}\otimes L_{\infty}(\mathbb{T}^{d+1})).$ This completes the proof.
\end{proof}

\section{Proof of Theorem \ref{doi bound} for the case of integral spectra}\label{Sect=IntSpec}

The next Theorem \ref{main lemma} provides the crucial connection between Calder\'on--Zygmund operators and commutator estimates. The equality \eqref{conjugate} should be understood as a transference to Schur multipliers argument. Note that here we have an exact equality \eqref{conjugate}, which we did not yet obtain in \cite{CPSZ2}.

\begin{thm}\label{main lemma} For every contraction $f:\mathbb{Z}^d\to\mathbb{Z}$ and for every collection of commuting self-adjoint operators $\mathbf{A}=\{A_k\}_{k=1}^d\subset\mathcal{M}$ with ${\rm spec}(A_k)\subset\mathbb{Z},$ we have
$$\|T_{f_{k_0}}^{\mathbf{A},\mathbf{A}}(V)\|_{1,\infty}\leq c(d)\|V\|_1,\quad V\in L_1(\mathcal{M}),\quad 1\leq k_0\leq d.$$
Here, $f_{k_0}$ is given by \eqref{fk def}.
\end{thm}
\begin{proof} Fix $1 \leq k_0 \leq d$.  The idea is to construct a bounded linear operator $S:L_1(\mathcal{M}\otimes L_{\infty}(\mathbb{T}^{d+1}))\to L_{1,\infty}(\mathcal{M}\otimes L_{\infty}(\mathbb{T}^{d+1}))$ (independent of $f$) and an isometric embedding $I:L_1(\mathcal{M})\to L_1(\mathcal{M}\otimes L_{\infty}(\mathbb{T}^{d+1})),$ $I:L_{1,\infty}(\mathcal{M})\to L_{1,\infty}(\mathcal{M}\otimes L_{\infty}(\mathbb{T}^{d+1}))$ (dependent on $f$) such that
\begin{equation}\label{conjugate}
S\circ I=I\circ T_{f_{k_0}}^{\mathbf{A},\mathbf{A}}.
\end{equation}

Fix a smooth function $\varsigma:[0,1]\to\mathbb{R}$ such that $\varsigma(u)=u,$ $u\in[\frac12,1]$ and $\varsigma(u)\geq\frac13,$ $u\in[0,\frac12].$ Define a smooth function $g:\mathbb{S}^d\to\mathbb{R}$ by setting
$$g(t)=\frac{t_{k_0}t_{d+1}}{\varsigma(\sum_{k=1}^dt_k^2)},\quad |t|_2=1.$$
Extend $g$ to a smooth homogeneous function $g:\mathbb{R}^{d+1}\backslash\{0\}\to\mathbb{R}$ (of degree $0$) by setting $g(t)=g(\frac{t}{|t|_2}),$ $0\neq t\in\mathbb{R}^{d+1}.$ For $|t|_2=1,$ the conditions $\sum_{k=1}^dt_k^2\geq\frac12$ and $|t_{d+1}|\leq(\sum_{k=1}^dt_k^2)^{\frac12}$ are equivalent. Hence,
\begin{equation}\label{g value}
g(t)=\frac{t_{k_0}t_{d+1}}{\sum_{k=1}^dt_k^2},\quad |t_{d+1}|\leq(\sum_{k=1}^dt_k^2)^{\frac12},\quad 0\neq t\in\mathbb{R}^{d+1}.
\end{equation}

By assumption, $A_k=\sum_{i_k\in\mathbb{Z}}i_kp_{k,i_k},$ where $\{p_{k,i_k}\}_{i_k\in\mathbb{Z}}$ are pairwise orthogonal projections such that $\sum_{i_k\in\mathbb{Z}}p_{k,i_k}=1.$ Since $A$ is bounded, it follows that $p_{k,i_k}=0$ for all but finitely many $i_k\in\mathbb{Z}.$ Hence, these sums are, in fact, finite.
For every $\mathbf{i}=(i_1,\cdots,i_d)\in\mathbb{Z}^d,$ set $p_{\mathbf{i}}=p_{1,i_1}\cdots p_{d,i_d}.$ It is immediate that $\{p_{\mathbf{i}}\}_{\mathbf{i}\in\mathbb{Z}^d}$ are pairwise orthogonal projections and $\sum_{\mathbf{i}\in\mathbb{Z}^d}p_{\mathbf{i}}=1.$ Consider a unitary operator
$$U_f=\sum_{\mathbf{i}\in\mathbb{Z}^d}p_{\mathbf{i}}\otimes e_{(\mathbf{i},f(\mathbf{i}))},$$
where $e_{(\mathbf{i},f(\mathbf{i}))}$ is given in \eqref{ek def}.

We are now ready to define the operators $S$ and $I.$ Set
$$S(W)=(1\otimes g(\nabla_{\mathbb{T}^{d+1}}))(\sum_{\substack{\mathbf{i},\mathbf{j}\in\mathbb{Z}^d\\ \mathbf{i}\neq\mathbf{j}}}(p_{\mathbf{i}}\otimes 1)W(p_{\mathbf{j}}\otimes 1)),\quad W\in L_1(\mathcal{M}\otimes L_{\infty}(\mathbb{R}^{d+1})),$$
$$I(V)=U_f(V\otimes 1)U_f^*,\quad V\in L_{1,\infty}(\mathcal{M}).$$

Since $f$ is a contraction we have that $|f(\mathbf{i})- f(\mathbf{j})| \leq | \mathbf{i}-\mathbf{j}|_2$ and therefore by \eqref{g value} we obtain
\[
g(\mathbf{i}-\mathbf{j}, f(\mathbf{i})- f(\mathbf{j}) ) =  f_{k_0}(\mathbf{i},\mathbf{j}), \qquad \mathbf{i}, \mathbf{j} \in \mathbb{Z}^{d}.
\]
 In particular
 \[
 g(\nabla_{\mathbb{T}^{d+1}}) e_{(\mathbf{i}-\mathbf{j},f(\mathbf{i})-f(\mathbf{j}))}) = f_{k_0}(\mathbf{i},\mathbf{j}) e_{(\mathbf{i}-\mathbf{j},f(\mathbf{i})-f(\mathbf{j}))}),  \qquad \mathbf{i}, \mathbf{j} \in \mathbb{Z}^{d}.
 \]
   Recall also that $f_{k_0}( \mathbf{i},  \mathbf{i}) = 0,  \mathbf{i} \in \mathbb{Z}^{d}$.
We now prove the transference equality \eqref{conjugate}:
\[
\begin{split}
S(I(V))= &  S  \left( \Big(\sum_{\mathbf{i}\in\mathbb{Z}^d}p_{\mathbf{i}}\otimes e_{(\mathbf{i},f(\mathbf{i}))}\Big)\cdot\Big(\sum_{\substack{\mathbf{i},\mathbf{j}\in\mathbb{Z}^d }}p_{\mathbf{i}}Vp_{\mathbf{j}}\otimes 1 \Big)\cdot \Big(\sum_{\mathbf{i}\in\mathbb{Z}^d}p_{\mathbf{i}}\otimes e_{(\mathbf{i},f(\mathbf{i}))} \Big)^\ast \right) \\
= &S(\sum_{\mathbf{i},\mathbf{j}\in\mathbb{Z}^d}p_{\mathbf{i}}Vp_{\mathbf{j}}\otimes e_{(\mathbf{i}-\mathbf{j},f(\mathbf{i})-f(\mathbf{j}))}) \\
\stackrel{\eqref{g value}}{=} & \sum_{\substack{\mathbf{i},\mathbf{j}\in\mathbb{Z}^d\\ \mathbf{i}\neq\mathbf{j}}}p_{\mathbf{i}}Vp_{\mathbf{j}}\otimes f_{k_0}(\mathbf{i},\mathbf{j}) e_{(\mathbf{i}-\mathbf{j},f(\mathbf{i})-f(\mathbf{j}))} \\
= & \Big(\sum_{\mathbf{i}\in\mathbb{Z}^d}p_{\mathbf{i}}\otimes e_{(\mathbf{i},f(\mathbf{i}))}\Big)\cdot\Big(\sum_{\substack{\mathbf{i},\mathbf{j}\in\mathbb{Z}^d }}p_{\mathbf{i}} T_{f_{k_0}}^{\mathbf{A},\mathbf{A}}(V) p_{\mathbf{j}}\otimes 1 \Big)\cdot \Big(\sum_{\mathbf{i}\in\mathbb{Z}^d}p_{\mathbf{i}}\otimes e_{(\mathbf{i},f(\mathbf{i}))} \Big)^\ast \\
=&U_f\cdot (T_{f_{k_0}}^{\mathbf{A},\mathbf{A}}(V)  \otimes 1) \cdot U_f^*=I(T_{f_{k_0}}^{\mathbf{A},\mathbf{A}}(V)).
\end{split}
\]

By Theorem \ref{axis torus}, the  mapping
$$1\otimes g(\nabla_{  \mathbb{T}^{d+1}   }):L_1(\mathcal{M}\otimes L_\infty(\mathbb{T}^{d+1}))\to L_{1,\infty}(\mathcal{M}\otimes L_\infty(\mathbb{T}^{d+1})).$$
is bounded.
Therefore,
\[
\begin{split}
&\|T_{f_{k_0}}^{\mathbf{A},\mathbf{A}}(V)\|_{L_{1,\infty}(\mathcal{M})}=\|I(T_{f_{k_0}}^{\mathbf{A},\mathbf{A}}(V))\|_{L_{1,\infty}(\mathcal{M}\otimes L_{\infty}(\mathbb{T}^{d+1}))} \\
=&\|S(I(V))\|_{L_{1,\infty}(\mathcal{M}\otimes L_{\infty}(\mathbb{T}^{d+1}))}\\
\leq & \|S\|_{L_1(\mathcal{M}\otimes L_{\infty}(\mathbb{T}^{d+1}))\to L_{1,\infty}(\mathcal{M}\otimes L_{\infty}(\mathbb{T}^{d+1}))}\|I(V)\|_{L_1(\mathcal{M}\otimes L_{\infty}(\mathbb{T}^{d+1}))}\\
\lesssim & \|1\otimes g(\nabla)\|_{L_1(\mathcal{M}\otimes L_{\infty}(\mathbb{T}^{d+1}))\to L_{1,\infty}(\mathcal{M}\otimes L_{\infty}(\mathbb{T}^{d+1}))}\|V\|_{L_1(\mathcal{M})}.
\end{split}
\]
This completes the proof.
\end{proof}

\section{Proof of the main results}

In this section we collect the results announced in the abstract and its corollaries.

\begin{lem}\label{riemann} Let $\mathbf{A}=\{A_k\}_{k=1}^d\subset\mathcal{M}$ be an arbitrary collection of commuting self-adjoint operators. If $\{\xi_n\}_{n\geq0}$ is a uniformly bounded sequence of Borel functions on $\mathbb{R}^{2d}$ such that $\xi_n\to\xi$ everywhere, then
\begin{equation}\label{riemann sum convergence}
T_{\xi_n}^{\mathbf{A},\mathbf{A}}(V)\to T_{\xi}^{\mathbf{A},\mathbf{A}}(V),\quad V\in L_2(\mathcal{M})
\end{equation}
in $L_2(\mathcal{M})$ as $n\to\infty.$
\end{lem}
\begin{proof} Let $\nu$ be a projection valued measure on $\mathbb{R}^{2d}$ considered in Subsection \ref{doi subsection} (see \eqref{birsol spectral measure}). Let $\gamma:\mathbb{R}\to\mathbb{R}^{2d}$ be a Borel measurable bijection. Clearly, $\nu\circ\gamma$ is a countably additive projection valued measure on $\mathbb{R}.$ Hence, there exists a self-adjoint operator $B$ acting on the Hilbert space $L_2(\mathcal{M})$ such that $E_B=\nu\circ\gamma.$

Set $\eta_n=\xi_n\circ\gamma$ and $\eta=\xi\circ\gamma.$ We have $\eta_n\to\eta$ everywhere on $\mathbb{R}.$ Thus,
\[
\begin{split}
   & T_{\xi_n}^{\mathbf{A},\mathbf{A}}=\int_{\mathbb{R}^{2d}}\xi_nd\nu=\int_{\mathbb{R}}\eta_n(\lambda)dE_B(\lambda)=\eta_n(B)\to\eta(B) \\
= & \int_{\mathbb{R}}\eta(\lambda)dE_B(\lambda)=\int_{\mathbb{R}^{2d}}\xi d\nu=T_{\xi}^{\mathbf{A},\mathbf{A}}.
\end{split}
\]
Here, the convergence is understood with respect to the strong operator topology on the space $B(L_2(\mathcal{M})).$ In particular, \eqref{riemann sum convergence} follows.
\end{proof}

In the next proof let $\lfloor x \rfloor$ be the largest integer smaller than $x$ and let $\{ x \} = x - \lfloor x \rfloor$ be the fractional part.

\begin{proof}[Proof of Theorem \ref{doi bound}] {\bf Step 1.} Let $f:\mathbb{R}^d\to\mathbb{R}$ be a contraction. We claim that the mapping $f^n:\mathbb{Z}^d\to\mathbb{Z}$ defined by the formula
$$f^n({\bf i})=\lfloor \frac{n}{2}f(\frac{{\bf i}}{n})\rfloor,\quad {\bf i}\in\mathbb{Z}^d,$$
is also a contraction.

Indeed, we have
$$f^n({\bf i})-f^n({\bf j})=\frac{n}{2}(f(\frac{{\bf i}}{n})-f(\frac{{\bf j}}{n}))+(\{\frac{n}{2}f(\frac{{\bf j}}{n})\}-\{\frac{n}{2}f(\frac{{\bf i}}{n})\}).$$
By assumption, we have that
$$\frac{n}{2}|f(\frac{{\bf i}}{n})-f(\frac{{\bf j}}{n})|\leq\frac{n}{2}|\frac{{\bf i}}{n}-\frac{{\bf j}}{n}|\leq\frac12|{\bf i}-{\bf j}|.$$
It is immediate that
$$\{\frac{n}{2}f(\frac{{\bf j}}{n})\}-\{\frac{n}{2}f(\frac{{\bf i}}{n})\}\in(-1,1).$$
Thus,
$$|f^n({\bf i})-f^n({\bf j})|<\frac12|{\bf i}-{\bf j}|+1.$$

If $|{\bf i}-{\bf j}|\geq2,$ then
$$|f^n({\bf i})-f^n({\bf j})|<\frac12|{\bf i}-{\bf j}|+1\leq|{\bf i}-{\bf j}|$$
and the claim follows. If $|{\bf i}-{\bf j}|<2,$ then
$$|f^n({\bf i})-f^n({\bf j})|<\frac12|{\bf i}-{\bf j}|+1<2.$$
Since $|f^n({\bf i})-f^n({\bf j})|\in\mathbb{N},$ it follows that
$$|f^n({\bf i})-f^n({\bf j})|\leq 1\leq |{\bf i}-{\bf j}|$$
provided that ${\bf i}\neq{\bf j}.$ This proves the claim for $|{\bf i}-{\bf j}|<2.$

{\bf Step 2.} Let $f:\mathbb{R}^d\to\mathbb{R}$ be a contraction. For every $n\geq1,$ set
$$A_{k,n}\stackrel{def}{=}\sum_{i_k\in\mathbb{Z}}i_kE_A([\frac{i_k}{n},\frac{i_k+1}{n})),\quad \mathbf{A}_n=\{A_{k,n}\}_{k=1}^d.$$
Fix $1 \leq k_0 \leq d$. Then
$$\xi_n(t,s)=(f^n)_{k_0}(\mathbf{i},\mathbf{j}),\quad t_k\in[\frac{i_k}{n},\frac{i_k+1}{n}),s_k\in [\frac{j_k}{n},\frac{j_k+1}{n}),\quad\mathbf{i},\mathbf{j}\in\mathbb{Z}^d.$$
It is immediate that (see e.g. Lemma 8  in \cite{PScrelle} for a much stronger assertion)
$$T_{\xi_n}^{\mathbf{A},\mathbf{A}}(V)=T_{(f^n)_{k_0}}^{\mathbf{A}_n,\mathbf{A}_n}(V).$$
It follows from Theorem \ref{main lemma} that
$$\|T_{\xi_n}^{\mathbf{A},\mathbf{A}}(V)\|_{1,\infty}\leq c(d)\|V\|_1.$$
Note that $\xi_n\to \frac12f_{k_0}$ everywhere. It follows from Lemma \ref{riemann} that
$$T_{\xi_n}^{\mathbf{A},\mathbf{A}}(V)\to T_{\frac{1}{2} f_{k_0}}^{\mathbf{A},\mathbf{A}}(V),\quad V\in L_2(\mathcal{M})$$
in $L_2(\mathcal{M})$ (and, hence, in measure --- see e.g \cite{PSW}) as $n\to\infty.$ Since the quasi-norm in $L_{1,\infty}(\mathcal{M})$ is a Fatou quasi-norm \cite{LSZ}, it follows that
$$\|T_{f_{k_0}}^{\mathbf{A},\mathbf{A}}(V)\|_{1,\infty}\leq c(d)\|V\|_1,\quad V\in (L_1\cap L_2)(\mathcal{M}).$$
\end{proof}

\begin{cor}\label{Cor=LpEsitmate} For every Lipschitz function $f:\mathbb{R}^d\to\mathbb{R}$ and for every collection $\mathbf{A}=\{A_k\}_{k=1}^d$ of bounded commuting self-adjoint operators, the operator $T_{f_k}^{\mathbf{A},\mathbf{A}}$ extends to a bounded operator from $L_p(\mathcal{M})$ to $L_p(\mathcal{M}),$ $1<p<\infty.$
\end{cor}
\begin{proof} By Theorem \ref{doi bound}, $T_{f_k}^{\mathbf{A},\mathbf{A}}$ extends to a bounded operator from $L_1(\mathcal{M})$ to $L_{1,\infty}(\mathcal{M})$ for every $1 \leq k \leq d$. Since also $T_{f_k}^{\mathbf{A},\mathbf{A}}:L_2(\mathcal{M})\to L_2(\mathcal{M}),$ it follows from real interpolation that $T_{f_k}^{\mathbf{A},\mathbf{A}}:L_p(\mathcal{M})\to L_p(\mathcal{M}),$ $1<p<2.$ Thus, $(T_{f_k}^{\mathbf{A},\mathbf{A}})^*:L_{\frac{p}{p-1}}(\mathcal{M})\to L_{\frac{p}{p-1}}(\mathcal{M}),$ $1<p<2.$ Since $f_k(s,t)=f_k(t,s),$ $s,t\in\mathbb{R}^d,$ it follows that $(T_{f_k}^{\mathbf{A},\mathbf{A}})^*=T_{f_k}^{\mathbf{A},\mathbf{A}}.$ In particular, $T_{f_k}^{\mathbf{A},\mathbf{A}}:L_{\frac{p}{p-1}}(\mathcal{M})\to L_{\frac{p}{p-1}}(\mathcal{M}),$ $1<p<2.$ This concludes the proof.
\end{proof}


\begin{lem}\label{eg psw lemma} If $A_k,B\in B(H),$ $1\leq k\leq d,$ are self-adjoint operators such that $[A_k,B]\in L_2(H),$ $1\leq k\leq d,$ then, for every Lipschitz function $f,$ we have
$$\sum_{k=1}^dT_{f_k}^{\mathbf{A},\mathbf{A}}([A_k,B])=[f(\mathbf{A}),B].$$
Here $f_k$ is given by \eqref{fk def}.
\end{lem}
\begin{proof} By definition of double operator integral given in Subsection \ref{doi subsection}, we have for any bounded Borel function on $\mathbb{R}^{2d}$,
\begin{equation}\label{doi mult}
T_{\xi_1}^{\mathbf{A},\mathbf{A}}T_{\xi_2}^{\mathbf{A},\mathbf{A}}=T_{\xi_1\xi_2}^{\mathbf{A},\mathbf{A}}.
\end{equation}
Let $\xi_{1,k}=f_k$ and let $\xi_{2,k}(\lambda,\mu)=\lambda_k-\mu_k$ when $|\lambda|_2,|\mu|_2\leq\sup_{1\leq k\leq d}\|A_k\|_{\infty},$ $\xi_{2,k}(\lambda,\mu)=0$ when $|\lambda|_2 >\sup_{1\leq k\leq d}\|A_k\|_{\infty}$ or $|\mu|_2 > \sup_{1\leq k\leq d}\|A_k\|_{\infty}.$ It is immediate that
$$(\sum_{k=1}^d\xi_{1,k}\xi_{2,k})(\lambda,\mu)=f(\lambda)-f(\mu),\quad \lambda,\mu\in\mathbb{R}^d, \: {\rm s.t. } \: \vert \lambda \vert_2, \vert \mu \vert_2 \leq \sup_{1\leq k\leq d}\|A_k\|_{\infty}.$$
If $p$ is a finite rank projection, then $pB\in L_2(H)$ and
$$T_{\sum_{k=1}^d\xi_{1,k}\xi_{2,k}}^{\mathbf{A},\mathbf{A}}(pB)=f(\mathbf{A})pB-pBf(\mathbf{A}),\quad T_{\xi_{2,k}}^{\mathbf{A},\mathbf{A}}(pB)=A_kpB-pBA_k,$$
Applying \eqref{doi mult} to the operator $pB\in L_2(H),$ we obtain
\begin{equation}\label{finite p}
\sum_{k=1}^dT_{f_k}^{\mathbf{A},\mathbf{A}}(A_kpB-pBA_k)=f(\mathbf{A})pB-pBf(\mathbf{A}).
\end{equation}

By Theorem 4.2 in \cite{Voiculescu}, there exists a sequence $p_l$ of finite rank projections such that $p_l\to1$ strongly and such that, for every $1\leq k\leq d,$ $[A_k,p_l]\to0$ as $l\to\infty$ in $L_d(H)$ for $d>1$ and in $L_2(H)$ if $d=1.$ In particular,
$$A_kp_lB-p_lB A_k=p_l[A_k,B]+[A_k,p_l]B\to [A_k,B],\quad l\to\infty,$$
in $L_d(H).$

By the preceding paragraph and Corollary \ref{Cor=LpEsitmate}, we have
\begin{equation}\label{lconv}
T_{f_k}^{\mathbf{A},\mathbf{A}}(A_kp_lB-p_l B A_k)\to T_{f_k}^{\mathbf{A},\mathbf{A}}([A_k,B]),\quad l\to\infty,
\end{equation}
in $L_d(H).$ On the other hand,
\begin{equation}\label{rconv}
f(\mathbf{A})p_lB-p_lBf(\mathbf{A})\to f(\mathbf{A})B-Bf(\mathbf{A}),\quad l\to\infty,
\end{equation}
in the strong operator topology. Substituting \eqref{lconv} and \eqref{rconv} into \eqref{finite p}, we conclude the proof.
\end{proof}

\begin{proof}[Proof of Theorem \ref{final theorem}] By assumption, $[A_k,B]\in L_1(H)\subset L_2(H).$ The first assertion follows by combining Lemma \ref{eg psw lemma} and Theorem \ref{doi bound}. Applying the first assertion to the operators
$$
A_k=
\begin{pmatrix}
X_k&0\\
0&Y_k
\end{pmatrix},\quad
B=
\begin{pmatrix}
0&1\\
1&0
\end{pmatrix},
$$
we obtain the second assertion.
\end{proof}

\begin{cor}\label{Cor=Normal}
For every Lipschitz function $f:\mathbb{C} \to\mathbb{R}$ and for every   normal operator $A \in B(H)$ and every $B \in B(H)$ such that $[A,B]\in L_1(H),$ we have
$$\|[f(A),B]\|_{1,\infty}\leq c(d)\|\nabla(f)\|_{\infty} \|[A,B]\|_1.$$
For every Lipschitz function $f:\mathbb{C} \to \mathbb{R}$ and for every  pair $X, Y \in B(H)$ of normal operators  such that $X-Y\in L_1(H),$ we have
$$\|f(X)-f(Y)\|_{1,\infty}\leq c(d)\|\nabla(f)\|_{\infty} \|X-Y\|_1.$$
\end{cor}
\begin{proof}
An operator $A$ is normal if and only it can be written as $A = A_1 + i A_2$ with $A_1$ and $A_2$ commuting self-adjoint operators. Identifying $\mathbb{C} \simeq \mathbb{R}^2$ we may see $f$ as a 2 real variable Lipschitz function, say $\widetilde{f}$, and this identification is compatible with spectral calculus, i.e. $f(A) = \widetilde{f}(A_1, A_2)$. Then the corollary is a direct consequence of the statements in Theorem \ref{final theorem}.
\end{proof}

\appendix

\section{Fej\'er's lemma}
In our proof we use a von Neumann-valued Fej\'er's lemma. As we could not find a reference to this type of vector valued case we prove it here for convenience of the reader.

We let $e_l, l \in \mathbb{Z}$ denote  the standard trigonometric functions on the torus. Let $\mathcal{E}$ be the conditional expectation $\mathcal{M}\otimes L^\infty(\mathbb{T}^{d+1}) \rightarrow \mathcal{M} \otimes 1$. For $k\in\mathbb{Z}^{d+1}_+,$ let
$$S_k(x)=\sum_{\substack{l\in\mathbb{Z}^{d+1}\\ -k\leq l\leq k}} \mathcal{E} (x (1\otimes e_l)^\ast) (1\otimes e_l).$$
For $n\in\mathbb{Z}_+,$ we set
$$A_n(x)=(n+1)^{-d-1}\sum_{\substack{k\in\mathbb{Z}_+^{d+1}\\ k\leq (n,\cdots,n)}}S_k(x).$$
Here, the order on $\mathbb{Z}_+^{d+1}$ is defined by $m \leq n$ if $m_j \leq n_j$ for all $1 \leq j \leq d+1$.

\begin{rmk}\label{Rmk=Fejer2}
It follows directly that for  $x \in L_2(\mathcal{M}\otimes\mathbb{T}^{d+1})$ we have $\Vert A_n x - x \Vert_2 \rightarrow 0$ as $n \rightarrow \infty$.
\end{rmk}

The assertion below is known as Fej\'er's lemma.

\begin{lem}\label{Lem=Fejer} We have $\Vert A_n(x)- x\Vert_1 \rightarrow 0$ for all  $x \in L_1(\mathcal{M}\otimes\mathbb{T}^{d+1})$ as $n\to\infty.$
\end{lem}
\begin{proof} We split the proof in steps.

\vspace{0.3cm}

\noindent {\bf Step 1.} We claim that
$$\|A_nx\|_1\leq\|x\|_1,\quad x\in L_1(\mathcal{M}\otimes\mathbb{T}^{d+1}),\quad n\geq0.$$
To see this fact, we identify the space $L_1(\mathcal{M}\otimes\mathbb{T}^{d+1})$ with the space of vector-valued functions $L_1(\mathbb{T}^{d+1},L_1(\mathcal{M})).$ We now write a pointwise equality
$$(A_n(x))(t)=\int_{\mathbb{T}^{d+1}}x(t+s)\Phi_n(s)ds,\quad s\in\mathbb{T}^{d+1}.$$
Here, $\Phi_n:\mathbb{T}^{d+1}\to\mathbb{R}$ is the Fej\'er kernel  possessing the following properties.
$$\Phi_n(s)\geq0,\quad \int_{\mathbb{T}^{d+1}}\Phi_n(s)ds=1.$$
Thus,
$$\|A_nx\|_1\leq\int_{\mathbb{T}^{d+1}}\|x(\cdot+s)\|_1\Phi_n(s)ds=\|x\|_1.$$

\vspace{0.3cm}

\noindent {\bf Step 2.} Fix $\epsilon>0$ and choose a projection $p\in\mathcal{M}$ such that $\tau(p)<\infty$ and such that $\|x' \|_1<\epsilon,$ where
$$x'  := x-(p\otimes 1)x(p\otimes 1).$$
Choose $y\in L_2(p\mathcal{M}p\otimes\mathbb{T}^{d+1})$ such that
$$\|y-(p\otimes 1)x(p\otimes 1)\|_1<\epsilon.$$
In particular, we have that $\|y-x\|_1<2\epsilon.$

We clearly have $A_n y\to y$ in $L_2(p\mathcal{M}p\otimes\mathbb{T}^{d+1}).$ Since $\tau(p)<\infty,$ it follows that $A_n y\to y$ in $L_1(p\mathcal{M}p\otimes\mathbb{T}^{d+1}).$ Thus, $A_n y\to y$ in $L_1(\mathcal{M}\otimes\mathbb{T}^{d+1}).$ Choose $N$ so large that $\|A_n y- y\|_1<\epsilon$ for $n>N.$ It follows from Step 1 that
\[
\begin{split}
& \|A_n x- x\|_1\leq\|A_n(x-y)\|_1+\|A_n y - y\|_1+\|x-y\|_1\leq 2\|x-y\|_1+\|A_n y-y\|_1\\
\leq & 4\epsilon+\|A_n y- y\|_1<5\epsilon,\quad n>N.
\end{split}
\]
Since $\epsilon>0$ is arbitrarily small, the assertion follows.
\end{proof}


\begin{thebibliography}{99}
\bibitem{APPS}
   Aleksandrov A., Peller V., Potapov D., Sukochev F.,
    \emph{Functions of normal operators under perturbations},
     Adv. Math. {\bf 226} (2011), no. 6, 5216--5251.


\bibitem{BiSo1} Birman M., Solomyak M., {\it Double Stieltjes operator integrals.} (Russian), Probl. Math. Phys., Izdat. Leningrad. Univ., Leningrad, (1966)  33--67. English translation in: Topics in Mathematical Physics, Vol. 1 (1967),  Consultants Bureau Plenum Publishing Corporation, New York, 25--54.
\bibitem{BirSol} Birman M., Solomyak M. {\it Spectral theory of selfadjoint operators in Hilbert space.} Mathematics and its Applications (Soviet Series). D. Reidel Publishing Co., Dordrecht, 1987.

 \bibitem{BirSol89}
   Birman M., Solomyak M.,
   \emph{Operator integration, perturbations and commutators (Russian)},
    translated from Zap. Nauchn. Sem. Leningrad. Otdel. Mat. Inst. Steklov. (LOMI) {\bf 170} (1989).

\bibitem{BirSolDOI}
    Birman M., Solomyak M.,
    \emph{Double operator integrals in a Hilbert space},
     Integral Equations and Operator Theory {\bf 47} (2003), no. 2, 131--168.
\bibitem{Cad}
   Cadilhac L., \emph{Weak boundedness of Calder\'on-Zygmund operators on noncommutative $L^1$-spaces},
  arXiv:1702.06536.
\bibitem{CMPS} Caspers M., Montgomery-Smith S., Potapov D., Sukochev F., {\it The best constants for operator Lipschitz functions on Schatten classes.} J. Funct. Anal. {\bf 267} (2014), no. 10, 3557--3579.
\bibitem{CPSZ} Caspers M., Potapov D., Sukochev F., Zanin D., {\it Weak type estimates for the absolute value mapping.} J. Operator Theory, {\bf 73} (2015), no. 2, 101--124.
\bibitem{CPSZ2} Caspers M., Potapov D., Sukochev F., Zanin D., {\it Weak type commutator and Lipschitz estimates: resolution of the Nazarov-Peller conjecture}, Amer. J. Math.  (to appear).
\bibitem{CM} Connes A., Marcolli M., \emph{Noncommutative geometry, quantum fields and motives}, American Mathematical Society Colloquium Publications, 55. American Mathematical Society, Providence, RI; Hindustan Book Agency, New Delhi, 2008. xxii+785 pp.
\bibitem{Davies} Davies E., {\it Lipschitz continuity of functions of operators in the Schatten classes.} J. Lond. Math. Soc. {\bf 37} (1988) 148--157.
\bibitem{DDPS1} Dodds P., Dodds T., de Pagter B., Sukochev F., {\it Lipschitz continuity of the absolute value and Riesz projections in symmetric operator spaces.}   J. Funct. Anal. {\bf 148} (1997), 28--69.
\bibitem{DDPS2} Dodds P., Dodds T., de Pagter B., Sukochev F., {\it Lipschitz continuity of the absolute value in preduals of semifinite factors.} Integral Equations Operator Theory {\bf 34} (1999), 28--44.
\bibitem{Far8} Farforovskaya Y., {\it An example of a Lipschitz function of self-adjoint operators with non-nuclear difference under a nuclear perturbation.} Zap. Nauchn. Sem. Leningrad. Otdel. Mat. Inst. Steklov. (LOMI) {\bf 30} (1972), 146--153.
\bibitem{Grafakos} Grafakos L., {\it Classical Fourier analysis.} Third edition. Graduate Texts in Mathematics, {\bf 249}. Springer, New York, 2014.

\bibitem{HNVW}
   Hyt\"onen T.,  van Neerven J.,  Veraar  M.,  Weis L.,
   \emph{Analysis in Banach spaces.
   Volume I: Martingales and Littlewood-Paley theory},
    Ergebnisse der Mathematik und ihrer Grenzgebiete. 2016.
\bibitem{KS} Kalton N., Sukochev F., {\it Symmetric norms and spaces of operators.} J. Reine Angew. Math. {\bf 621} (2008), 81--121.
\bibitem{Kato} Kato T., {\it Continuity of the map $S \mapsto |S|$ for linear operators.} Proc. Japan Acad. {\bf 49} (1973) 157--160.
\bibitem{Kosaki} Kosaki H., {\it Unitarily invariant norms under which the map $A \mapsto \vert A\vert$ is continuous.} Publ. Res. Inst. Math. Sci. {\bf 28} (1992), 299--313.
\bibitem{Krein} Krein M., {\it Some new studies in the theory of perturbations of self-adjoint operators.} First Math. Summer School, Part I (Russian), Izdat. \lq\lq Naukova Dumka\rq\rq, Kiev, 1964, pp. 103--187.
\bibitem{KPSS} Kissin E., Potapov D., Shulman V., Sukochev F.,
  {\it Operator smoothness in Schatten norms for functions of several variables: Lipschitz conditions, differentiability and unbounded derivations},
   Proc. Lond. Math. Soc. (3) {\bf 105} (2012), no. 4, 661--702.
\bibitem{de Leeuw} de Leeuw K., {\it On $L_p$ multipliers}. Ann. of Math. (2) {\bf 81 } (1965) 364--379.

\bibitem{LSZ} Lord S., Sukochev F., Zanin D., {\it Singular traces. Theory and applications.} De Gruyter Studies in Mathematics, {\bf 46}. De Gruyter, Berlin, 2013.
\bibitem{NazarovPeller} Nazarov F., Peller V., {\it Lipschitz functions of perturbed operators.} C. R. Math. Acad. Sci. Paris {\bf 347} (2009), 857--862.




 \bibitem{PSJFA04}
  de Pagter B., Sukochev F.,
   \emph{Differentiation of operator functions in non-commutative $L_p$-spaces},
    J. Funct. Anal. {\bf 212} (2004), no. 1, 28--75.

\bibitem{PSW}  de Pagter B., Sukochev F., Witvliet H., {\it Double operator integrals}, J. Funct. Anal. {\bf 192} (2002), no. 1, 52--111.
\bibitem{Parcet} Parcet J., {\it Pseudo-localization of singular integrals and noncommutative Calder\'on-Zygmund theory.} J. Funct. Anal. {\bf 256} (2009), 509--593.
\bibitem{PellerHankel} Peller V., {\it Hankel operators in the theory of perturbations of unitary and selfadjoint operators.} Funktsional. Anal. i Prilozhen. {\bf 19} (1985), no. 2, 37--51, 96.

\bibitem{PisierBook}
 Pisier G.,
  \emph{Introduction to operator space theory},
   London Mathematical Society Lecture Note Series, 294. Cambridge University Press, Cambridge, 2003. viii+478 pp.


\bibitem{PotapovSukochev2} Potapov D., Sukochev F., {\it Operator-Lipschitz functions in Schatten-von Neumann classes.} Acta Math. {\bf 207} (2011), no. 2, 375--389.
\bibitem{PScrelle} Potapov D., Sukochev F., {\it Unbounded Fredholm modules and double operator integrals.} J. Reine Angew. Math. {\bf 626} (2009), 159--185.
\bibitem{Skripka}
  Skripka A.,
  \emph{Asymptotic expansions for trace functionals},
  J. Funct. Anal. {\bf 266} (2014), no. 5, 2845--2866.

\bibitem{Stein} Stein E., {\it Singular integrals and differentiability properties of functions.} Princeton Mathematical Series, No. {\bf 30} Princeton University Press, Princeton, N.J.
\bibitem{Sind} Sukochev F., {\it Completeness of quasi-normed symmetric operator spaces.} Indag. Math. (N.S.) {\bf 25} (2014), no. 2, 376--388.
\bibitem{Suij}
   van Suijlekom W.,    \emph{Perturbations and operator trace functions},  J. Funct. Anal. {\bf 260} (2011), no. 8, 2483--2496.
\bibitem{Voiculescu} Voiculescu D., {\it Some results on norm-ideal perturbations of Hilbert space operators.} J. Operator Theory {\bf 2} (1979), no. 1, 3--37.
\end{thebibliography}
\end{document}